\documentclass[a4paper,11pt]{amsart}

\usepackage{amstext,amsfonts,amsthm,graphicx,amssymb,amscd,epsfig}
\usepackage{amsmath}
\usepackage[utf8]{inputenc}    
\usepackage[all]{xy}
\usepackage{mathrsfs}
\usepackage{mathtools}
\usepackage{enumitem}
\usepackage[capitalise]{cleveref}

\theoremstyle{definition}
\newtheorem{definition}{Definition}[section]
\newtheorem{example}[definition]{Example}
\newtheorem{remark}[definition]{Remark}

\theoremstyle{plain}
\newtheorem{prop}[definition]{Proposition}
\newtheorem{lemma}[definition]{Lemma}
\newtheorem{coro}[definition]{Corollary}
\newtheorem{conjecture}[definition]{Conjecture}
\newtheorem{theorem}[definition]{Theorem}

\usepackage{./custco}

\title[Characterization of quasi-homogeneous map-germs]{A characterization of quasi-homogeneity in terms of liftable vector fields}

\author{I. Breva Ribes, R. Oset Sinha}

\date{}

\address{Departament de Matem\`atiques,
Universitat de Val\`encia, Campus de Burjassot, 46100 Burjassot,
Spain}

\email{raul.oset@uv.es}

\email{ignacio.breva@uv.es}

\thanks{Work of both authors partially supported by Grant PID2021-124577NB-I00 funded by MCIN/AEI/ 10.13039/501100011033 and by ``ERDF A way of making Europe".
Work of I. Breva Ribes supported by grant UV-INV-PREDOC22-2187086, funded by Universitat de València}


\begin{document}
\begin{abstract}
We prove under certain conditions that any stable unfolding of a quasi-homogeneous map-germ with finite singularity type is substantial.
We then prove that if an equidimensional map-germ is finitely determined, of corank 1, and either it admits a minimal stable unfolding or it is of multipliticy 3, then it admits a substantial unfolding if and only if it is quasi-homogeneous. 
Based on this we pose the following conjecture: a finitely determined map-germ is quasi-homogeneous if and only if it admits a substantial unfolding.
\end{abstract}

\maketitle

\section{Introduction}

Quasi-homogeneity has played an important role in the study of singularities for a long time.
For instance, many authors have given formulas to compute certain numerical topological invariants in terms of the weights and degrees of singular quasi-homogeneous map-germs (see \cite{greuel_formula,ohmoto_formula,pallarespenafort_formula} amongst others), making quasi-homogeneity a very desirable property.

A germ of a function $g\colon(\bbc^n,0)\to(\bbc,0)$ is weighted-homogeneous if there exist some weights $w_1,\ldots,w_n\in\bbn$ and degree $d\in\bbn$ such that for each $\lambda\in\bbc$
$$g(\lambda^{w_1}x_1,\ldots,\lambda^{w_n}x_n) = \lambda^d g(x_1,\ldots,x_n).$$
It is quasi-homogeneous if it is weighted-homogeneous after some coordinate change. These functions are crucial in the study of the relation between the Milnor and Tjurina numbers
$$\mu(g) = \dim_\bbc\frac{\ofu_n}{Jg}, \; \tau(g) = \dim_\bbc\frac{\ofu_n}{Jg + \langle g\rangle}$$
where $Jg$ is the ideal generated by the partial derivatives of $g$. 
Both of these are invariants of $g$: the first one is topological and measures the number of $n$-spheres in the homotopy type of the Milnor fibre of $g$, and the second one is analytic and measures the minimal number of parameters needed for a versal unfolding of $g$.

When $g$ has isolated singularity both invariants are finite and from their algebraic description it is immediate that $\mu(g) \geq \tau(g)$.
In particular, when $g$ is weighted-homogeneous the Euler relation is satisfied:
$$ w_1x_1\dpar{g}{x_1}(x) + \cdots +  w_nx_n\dpar{g}{x_n}(x) = d\cdot g(x).$$
Hence, quasi-homogeneous functions satisfy $g \in Jg$ and $\mu(g) = \tau(g)$.

In 1971 Saito proved the converse implication (\cite{saito}) and so if $g$ has isolated singularity, then $g$ is quasi-homogeneous if and only if $\mu(g) = \tau(g)$.
For a great compilation of some of Saito's techniques we refer to \cite{bkr_bruce}; here the authors study pairs $(f,X)$ of germs of a function and a variety in $\bbc^n$ and give sufficient conditions to determine when the equality of two invariants associated to the pair, called the relative Milnor and Tjurina numbers, characterizes the fact that a pair $(f,X)$ of germs of a function and variety in $\bbc^n$ share a common coordinate system in which both are weighted-homogeneous with respect to the same weights.
%



There are analogous definitions of the Milnor and Tjurina numbers of an isolated complete intersection singularity (ICIS) but in this case their algebraic descriptions are more complicated than in the hypersurface setting, making the problem of determining wether $\mu(X,0)\geq \tau(X,0)$ much harder.
Greuel proved in \cite{greuel_mutau} that when $(X,0)$ is quasi-homogeneous the equality $\mu(X,0) = \tau(X,0)$ is satisfied.
The opposite implication came much later: it is due to Vosegaard and can be found in \cite{vosegaard_mutau}.
The general inequality was completed in the meantime, but it required several steps.
First, in the mentioned article, Greuel also proved that the inequality holds for ICIS of dimension 1 and other cases.
Looijenga then proved it for the case that $X$ is of dimension 2 in \cite{looijenga_mutau2} (as cited in \cite{looijengasteenbrink}).
Finally, Looijenga and Steenbrink proved the general case in \cite{looijengasteenbrink}.

\medskip

The equivalent problem for the case of map-germs is still open.
A map-germ $f\colon(\bbc^n,0)\to(\bbc^p,0)$ is weighted-homogeneous if all of its components are weighted-homogeneous with respect to the same weights.
The role of the Tjurina number is here played by the $\A_e$-codimension, denoted by $\aecod (f)$.
The Milnor number is generalized by the discriminant Milnor number, $\mu_\Delta(f)$, when $n\geq p$ and by the image Milnor number, $\mu_I(f)$, when $n < p$, which are defined by taking a stabilization and looking at the homotopy type of the discriminant in the first case and that of the image in the second.
It was proven by Damon and Mond in \cite{damonmond_mudisc} that when $n\geq p$
$$\mu_\Delta(f) \geq \aecod (f)$$
with equality if $f$ is weighted-homogeneous in some coordinate system.
The opposite implication for the equality is unknown.
In the case that $p=n+1$ and $(n,n+1)$ is in the nice dimensions (i.e $n < 15$, see section 5.2 in \cite{nunomond}), the Mond conjecture states that the corresponding inequality
$$\mu_I(f)\geq \aecod(f)$$
is also satisfied, with equality if $f$ is quasi-homogeneous.
The opposite implication for the inequality is also unknown, although the conjecture is usually stated with just one implication (see section 2.6.2 in \cite{handbookIII}).
The conjecture has been solved when $n=1,2$ (see \cite{mondbentwires,vanstraten,mondconj}) but is still open in general.
The main obstruction is the lack of an explicit algebraic description of the image Milnor number, although some candidates have been proposed for this, see \cite{bobadillanunopenafort}.

\medskip


\medskip

In this paper, we give a necessary condition for any $\K$-finite (and, in particular, $\A$-finite) map-germ with stable unfolding in the nice dimensions to be weighted-homogeneous after an analytic coordinate change in source and target.
Moreover, we see that this condition characterizes quasi-homogeneity in the case that $n=p$ and $f$ is of corank 1 with minimal stable unfolding or if it has multiplicity $3$.

This condition is given in terms of the stable unfolding of $f$, and is a generalization of the concept of substantial 1-parameter stable unfoldings.
Essentially, if $F\colon(\bbc^n\times\bbc,0)\to(\bbc^p\times\bbc,0)$ is a 1-parameter stable unfolding of $f$, using $(X,\Lambda)$ for the coordinates in $\bbc^p\times\bbc$ we say that $F$ is {\it substantial} if $\Lambda\in d\Lambda(\Lift(F))$, where $\Lift(F)$ is the set of germs of vector fields in the target $\eta\in\theta_{p+1}$ such that there is a vector field in the source $\xi\in\theta_{n+1}$ satisfying $\eta\circ F = dF(\xi)$.
We say $\eta$ is liftable and $\xi$ lowerable for $F$.

Substantiality was originally defined in \cite{houstonaug2} as a technical condition that helps in computing the $\A_e$-codimension of certain map-germs called augmentations.
The study of these unfoldings led in \cite{phiequiv} to using the notion of $\lambda$-equivalent unfoldings, a relation which preserves substantiality and the $\A$-equivalence classes of augmentations of singularities by quasi-homogeneous functions.

Here we propose a generalization of this idea for stable unfoldings with more parameters, which specializes to the same definition in the $1$-parameter case.
After the examination of multiple examples and the results presented in this paper, we believe that this generalized property goes beyond its use in augmentation of singularities and propose the following conjecture:

\begin{conjecture}
Let $f\colon(\bbc^n,0)\to(\bbc^p,0)$ be an $\A$-finite map-germ such that its stable unfolding with minimal number of parameters lies in the nice dimensions.
Then, $f$ is quasi-homogeneous if and only if it admits a substantial unfolding.
\end{conjecture}

The core idea is that by looking at some subset of the eigenvalues of the matrix corresponding to the $1$-jet of a liftable vector field (in particular, looking at the eigenvalues of the projection over the parameter space), one can determine the rest of the eigenvalues of the liftable and lowerable vector fields, essentially showing if they can be converted into Euler vector fields in some coordinate system, in which case one can obtain a weighted-homogeneous normal form of the map-germ.

One of our main results supporting this conjecture is \cref{thm_wh_imp_subs}, which shows that every quasi-homogeneous map-germ must admit a whole family of substantial unfoldings (in fact, all of them are what we will call {\it weak substantial}).
As we will see, this is already useful to discard if a certain map-germ is quasi-homogeneous.
Notice that in order to use the Mond conjecture to determine if a map-germ is not weighted-homogeneous, one needs to compute both the image Milnor number and the $\A_e$-codimension and check if they differ.
In order to obtain $\A_e$-codimension, one of the most direct methods requires to compute the liftable vector fields of the stable unfolding (see \cite{damonakv}), so once this is done it is easier to check if the unfolding is substantial than to compute the image Milnor number.

The structure of this paper is as follows: in \cref{section_preliminaries} we introduce the general concept of substantiality, along with an auxiliar weak substantiality property, and the equivalence relations that preserve both of them, as well as all the definitions and previous results, and some minor results which are easily deduced from the definitions.
In \cref{section_necessary}, we prove that every weighted-homogeneous map under certain conditions admits a whole family of unfoldings which are substantial. This condition is satisfied in particular when $f$ admits a stable unfolding in the nice dimensions.
Then in \cref{section_converse} we show the converse for the case of corank 1, equidimensional map-germs with minimal stable unfolding or with multiplicity 3.
All of the sections contain multiple examples that illustrate both the definitions and the implications of the results.

\section{Preliminaries} \label{section_preliminaries}

We will work over $\bbk = \bbc, \bbr$ indistinctly unless otherwise specified. 
The ring of germs of smooth functions in $\bbk^s$ will be denoted by $\ofu_s$, $\mfr_s$ will be the maximal ideal given by functions that vanish at the origin, and the $\ofu_s$-module of germs of smooth vector fields in $\bbk^s$ will be denoted by $\theta_s$. 
If $\mononp f$ is a smooth map-germ, then $\theta(f)$ will be the set of vector fields along $f$, which can be identified with $\ofu_n^p$.
Recall that $f$ is of finite singularity type if it is $\K_e$-finite, i.e.  if the $\ofu_n$-module
$$\tke f = tf(\theta_n) + f^*\m_p\theta(f)$$
has finite $\bbk$-codimension in $\theta(f)$. 
Here $tf(\xi) = df(\xi)$ for all $\xi\in\theta_n$.
Similarly, $f$ is $\A$-finite if the module
$$\tae f = tf(\theta_n) + wf(\theta_p)$$
has finite $\bbk$-codimension in $\theta(f)$.
Here $wf(\eta) = \eta\circ f$ for all $\eta\in\theta_p$.
This codimension is denoted by $\aecod (f)$ and is related to the  following equivalence of map-germs: $\mononp{f,g}$ are $\A$-equivalent if there are germs of diffeomorphisms $\psi,\phi$ such that $\psi\circ f = g\circ \phi$.
A map-germ is {\it stable} if $\aecod(f) = 0$.

Notice that $wf(\theta_p)$ is an $\ofu_p$-module via $f$, but cannot be seen as an $\ofu_n$-module in any way, while $tf(\theta_n)$ is both an $\ofu_n$-module and an $\ofu_p$-module via $f$.
Hence, $\tae f$ is an $\ofu_p$-module via $f$. 

\begin{definition}\label{def_frelated}
Two vector fields $\eta\in\theta_p$ and $\xi\in\theta_n$ are {\it $f$-related} if the following equation is satisfied:
\begin{equation*}
\eta \circ f = df(\xi).
\end{equation*}
In this case it is said that $\eta$ is a {\it liftable} vector field of $f$, and that $\xi$ is a {\it lowerable} vector field of $f$.
Denote by $\Lift(f)$ the set of liftable vector fields of $f$, and by $\Low(f)$ the set of lowerable vector fields of $f$.
Both sets have a natural structure as $\ofu_p$-modules, the second one via $f$.

The {\it analytic stratum} of $f$ is defined as the $\bbk$-vector space $\tilde\tau(f) =\operatorname{ev}_0(\Lift(f))$, with $\operatorname{ev}_0$ being the evaluation at $0$ of vector fields in $\theta_p$.
A stable map-germ $f$ is said to be minimal if $\dim_\bbk\tilde\tau(f) = 0$
\end{definition}

The following lemma can be found as Lemma 6.1 in \cite{nishimuralifts}: 

\begin{lemma}\label{lem_nishimura_lifts}
Let $f,g\colon(\bbk^n,0)\to(\bbk^p,0)$ be smooth map-germs and assume there are some germs of diffeomorphism $\phi\colon(\bbk^n,0)\to(\bbk^n,0)$ and $\psi\colon(\bbk^p,0)\to(\bbk^p,0)$ such that $\psi\circ f\circ \phi = g$.
Then, the map
\begin{align*}
\Lift(f)&\to\Lift(g)\\
\eta &\mapsto d\psi \circ \eta \circ \psi^{-1}
\end{align*}
is a bijection. In particular, it is an isomorphism of $\ofu_p$-modules via $\psi^{-1}$.
\end{lemma}

\begin{definition}
If $H = (H_1,\ldots, H_m)\colon(\bbk^p,0)\to(\bbk^m,0)$ is a set of equations, the set of vector fields in $\theta_p$ which are tangent to $H^{-1}(0)$ is
$$\Derlog(H^{-1}(0)) = \left\lbrace \eta\in \theta(p) : \eta(\langle H_1,\ldots,H_m \rangle) \subseteq \langle  H_1,\ldots,H_m\rangle \right\rbrace$$
where $\eta$ acts over a function $\tilde H$  by $\eta(\tilde H) = \sum_{j=1}^p \eta_j(X) \dpar{\tilde H}{X_j}(X)$.
\end{definition}

Let $\Delta f$ be the discriminant of $f$, i.e., the image of the set of singular points of $f$ when $n \geq p$ or the image of $f$ when $n < p$.
For the following proposition, which only works for $\bbk = \bbc$, we refer to Proposition 8.8 and Remark 8.2 in \cite{nunomond}, and the remark after Definition 1 in \cite{nunoosetlifts}.

\begin{prop}\label{prop_lift_eq_derlog}
Let $f\colon(\bbc^n,0)\to(\bbc^p,0)$ smooth and assume either that $f$ is stable, or that it is $\A$-finite and $(n,p)$ do not satisfy $n>p\leq 2$.
Then $\Lift(f) = \Derlog(\Delta f)$.
In particular, this holds when $n=p$ and $f$ is $\A$-finite.
\end{prop}

%
%

\begin{definition}\label{def_weighted_homogeneous}
We say that the map-germ $\mononp{f = (f_1,\ldots,f_p)}$ is {\it weighted-homogeneous} if there exist $w_1,\ldots,w_n,d_1,\ldots,d_p$ positive integers such that 
$$f_j(\lambda^{w_1}x_1,\ldots,\lambda^{w_n}x_n) = \lambda^{d_j}f_j(x_1,\ldots,x_n)$$
for every $j= 1,\ldots,p$ and every $\lambda\in\bbk$.
In the case that $f$ is analytic this condition is equivalent to the fact that
$$d_j f_j(x) = \sum_{i=1}^n \dpar{f_j}{x_i}(x)x_iw_i$$
for every $j = 1,\ldots,p$.
We call $w_i$ the weight of the variable $x_i$ and $d_j$ the weighted degree of the component $f_j$.
Finally, we say that $f$ is quasi-homogeneous if it is $\A$-equivalent to a weighted-homogeneous map-germ.
\end{definition}

When $f$ is analytic this definition can be rewritten in terms of $f$-related vector fields: $f$ is weighted-homogeneous if and only if there exist Euler vector fields $\eta\in \Lift(f)$ and $\xi\in\Low(f)$ such that
\begin{align*}
\eta(X) &= \sum_{j=1}^p d_jX_j\dpar{}{X_j}\\
\xi(x) &= \sum_{i=1}^n w_ix_i\dpar{}{x_i}
\end{align*}
that are $f$-related.
Here $\dpar{}{X_j}$ and $\dpar{}{x_i}$ are the constant vector fields in $\theta_p$ and $\theta_n$.

Let $\hat \ofu_p$ be the formal completion of $\ofu_p$, which can be identified with the space of formal power series $\bbc[[x_1,\ldots,x_p]]$.
Given two vector fields $\eta,\eta'\in\theta_p$, denote its Lie bracket by $[\eta,\eta']$.

From now on when we refer to the {\it eigenvalues of a vector field}, we are referring to those of the matrix associated to the linear map determined by its 1-jet, $j^1\eta$.

\begin{theorem}\label{thm_poincare_dulac}
Let $\eta\in\theta_p$ be a germ of analytic vector field which vanishes at the origin, and let $d_1,\ldots,d_p\in\bbc$ be its eigenvalues.
Then:
\begin{enumerate}
\item There exists a formal diffeomorphism $\psi\in\hat\ofu_p^p$ such that $d\psi\circ \eta\circ\psi^{-1
}$ can be expressed as the sum $\eta_S + \eta_N$ of vector fields, with $\eta_S = \sum_{j=1}^p d_jX_j\dpar{}{X_j}$ and the linear part of $\eta_N$ being a nilpotent matrix, satisfying $[\eta_S,\eta_N] =0$.
This is called the Poincaré-Dulac normal form of $\eta$.
\item If $h\in\ofu_p$ satisifies $\eta(h) = \beta h$ for some $\beta\in\bbc$, then for $\bar h$ denoting $h$ in the coordinates given by the Poincaré-Dulac normal form:
\begin{align*}
\eta_S(\bar h) &= \beta \bar h\\
\eta_N(\bar h) &= 0.
\end{align*}
\item If $I$ is an ideal in $\ofu_p$ and $\eta(I)\subseteq I$, then for $\bar I$ denoting $I$ in the coordinates given by the Poincaré-Dulac normal form:
\begin{align*}
\eta_S(\bar I) &\subseteq I\\
\eta_N(\bar I) &\subseteq I.
\end{align*}
\end{enumerate}

Moreover, if $(d_1,\ldots,d_p)$ lies in the Poincaré domain (i.e., if $0$ does not belong to the convex hull in $\bbc$ of $(d_1,\ldots,d_p$)), then we can assume the diffeomorphism $\phi$ is analytic.
In particular, if $(d_1,\ldots,d_p)$ is a vector of positive integers, then it lies in the Poincaré domain.
\end{theorem}

We also need the following result about preservation of eigenvalues after coordinate changes:

\begin{lemma}\label{lemma_eigenvalues_preserve}
Let $\eta\in\theta_p$ and $\psi\colon(\bbk^p,0)\to(\bbk^p,0)$ a germ of diffeomorphsim, then $d\psi\circ\eta\circ\psi^{-1}$ has the same eigenvalues as $\eta$.
\end{lemma}
\begin{proof}
This is clear from looking at the matrices associated to the the $1$-jets of $\psi,\eta$ and $\psi^{-1}$.
\end{proof}

From here on $\mononp f$ will be an analytic map-germ of finite $\K_e$-codimension unless otherwise specified. 
We will use $x = (x_1,\ldots,x_n)$ to denote the variables in $\bbk^n$ and $X = (X_1,\ldots, X_p)$ for the variables in $\bbk^p$.
An $m$-parameter unfolding of $f$ is any map-germ $F\colon(\bbk^n\times\bbk^m,0)\to(\bbk^p\times\bbk^m,0)$ of the form $F(x,\lambda) = (f_\lambda(x),\lambda)$ such that $f_0 \equiv f$.
When $m=1$ and $F$ is stable, we say it is an OPSU (one-parameter stable unfolding).

We will use $(x,\lambda) = (x_1,\ldots,x_n,\lambda_1,\ldots,\lambda_{m})$ for the variables in $\bbk^n\times\bbk^m$, and $(X,\Lambda) = (X_1,\ldots,X_p,\Lambda_1,\ldots,\Lambda_m)$ for the variables in $\bbk^p\times\bbk^m$.
With this notation, $\m_\Lambda$ will be the ideal in $\ofu_{p+m}$ generated by the functions $\Lambda_1,\ldots,\Lambda_m$, and similarly with $\m_\lambda$ for the parameters in the source.

All map-germs of finite $\K_e$-codimension admit a stable unfolding, see \cite{nunomond}.
This motivates the following definition:

\begin{definition}\label{def_minimal_size_unfolding}
An $m$-parameter stable unfolding $F\colon (\bbk^n\times\bbk^m,0)\to(\bbk^p\times\bbk^m,0)$ of $f$ {\it has the minimal number of parameters} if $f$ admits no stable unfolding with less parameters.
\end{definition}

It is well-known that the minimal number of parameters that a map-germ requires to obtain a stable unfolding is given by the codimension as $\bbk$-vector space of the quotient of $\theta(f)$ by the module $\tke f + \tae f$.
As a reference, see \cite{notaopsus} or \cite{nunomond}.

\begin{definition}
If $F\colon (\bbk^n\times\bbk^m,0)\to(\bbk^p\times\bbk^m,0)$ is an unfolding of $f$, we say that a liftable vector field $\eta\in\Lift(F)$ is {\it projectable} if $d\pi(\eta) \in \m_\Lambda\theta(\pi)$, where $\pi\colon (\bbk^p\times\bbk^m,0)\to(\bbk^m,0)$ is the natural projection.
A similar definition follows for lowerable vector fields.
\end{definition}

If $\eta\in\Lift(F)$ and $\xi\in\Low(F)$ are $F$-related vector fields, then the equation $\eta\circ F = dF(\xi)$ implies that if $\eta$ is projectable, then $\xi$ is projectable, since the projection over the last $m$-components of the equation just gives $\xi_{p+k} = \eta_{p+k}\circ F$ for $k=1,\ldots,m$.

Moreover, if both $\eta$ and $\xi$ are projectable then the vector fields defined by $\tilde \eta(X) =\eta(X,0)$ and $\tilde \xi(x) = \xi(x,0)$, which can be seen as vector fields in $\theta_p$ and $\theta_n$ respectively and satisfy the equation $\tilde\eta \circ f = tf(\tilde \xi)$, are called their projections. 

\begin{example}\label{ex_unfolding_cusp}
The map-germ $f\colon(\bbk,0)\to(\bbk^2,0)$ given by $f(x) = (x^2,x^3)$ admits the 1-parameter stable unfolding $F(x,\lambda) = (x^2,x^3+\lambda x,\lambda)$.
The vector fields
\begin{align*}
\eta(X_1,X_2,\Lambda) &= 2X_1\dpar{}{X_1} + 3X_2\dpar{}{X_2}+ 2\Lambda\dpar{}{\Lambda}\\
\xi(x,\lambda) &= x\dpar{}{x}+2\lambda\dpar{}{\lambda}
\end{align*}
are $F$-related and projectable.
Their projections $\tilde\eta(X_1,X_2) = 2X_1\dpar{}{X_1} + 3X_2\dpar{}{X_2}$ and $\tilde\xi(x) = x\dpar{}{x}$ are $f$-related.
\end{example}

Projectable vector fields of a stable unfolding are in correspondence with the liftable and lowerable vector fields of the unfolded map-germ.
In fact, any liftable vector field of $\mononp f$ can be seen as the projection of a liftable vector field of its stable unfolding.
This was studied in \cite{nishimuralifts} and \cite{nunoosetlifts}, where the following result is obtained:

\begin{prop}\label{prop_lifts_unfolding}
Let $\mononp f$ be a non-stable germ of finite singularity type, and let $F$ be a stable $m$-parameter unfolding.
Then
$$\Lift(f) = \pi_1\left(i^*(\Lift(F)\cap \mathcal M)\right)$$
where $\mathcal M$ is the submodule of $\theta_{p+m}$ generated by the constants $\dpar{}{X_j}$ for $1\leq j \leq p$ and by the vector fields $\Lambda_k\dpar{}{\Lambda_j}$ for $1\leq k,j \leq m$, $\pi_1$ is the projection onto the first $p$ components and $i^*$ is the morphism induced by $i(X) = (X,0)$.
\end{prop}

\begin{remark}\label{rem_eigenvalues_projection_subset}
Notice that if $\eta\colon(\bbk^p\times\bbk^m,0)\to(\bbk^p\times\bbk^m,0)$ is a projectable, liftable vector field of a stable unfolding $F$ which projects to some $\eta_0$, then the eigenvalues of $\eta_0$ are a subset of the eigenvalues of $\eta$, since the $1$-jet of $\eta$ is a block-triangular matrix of the form
\begin{equation*}
\left(\begin{matrix}
A & B \\ 0 & C
\end{matrix}\right)
\end{equation*}
and the $1$-jet of $\eta_0$ is just the $p\times p$ submatrix $A$.
A similar thing happens with lowerable, projectable vector fields.
\end{remark}

\begin{definition}\label{def_substantial}
We say that the stable unfolding $F\colon(\bbk^n\times\bbk^m,0)\to(\bbk^p\times\bbk^m,0)$ of $\mononp f$ is {\it substantial} if there exists $\eta\in\Lift(F)$ such that for every $k=1,\ldots,m$
$$d\Lambda_k(\eta) = \Lambda_k u_k(X,\Lambda) + q_k(X,\Lambda)$$
where $q_k\in\mfr_\Lambda\mfr_{p+m}$ and $u_k\in\ofu_{p+m}$ is a unit.
Any such vector field $\eta$ will be called a {\it substantial vector field}.
\end{definition}

\begin{example}\label{ex_wh_one_substantial}
If $\mononp f$ admits a weighted-homogeneous unfolding, $F\colon(\bbk^n\times\bbk^m,0)\to(\bbk^p\times\bbk^m,0)$, then $F$ admits an Euler, liftable vector field
$$\eta(X,\Lambda) = d_1X_1\dpar{}{X_1}+\cdots+d_pX_p\dpar{}{X_p}+d_{p+1}\Lambda_1\dpar{}{\Lambda_1}+\cdots+d_{p+m}\Lambda_m\dpar{}{\Lambda_m}$$
which satisfies that $d\Lambda_k(\eta) = d_{p+k}\Lambda_k$ for each $1\leq k \leq m$, hence $F$ is substantial.
This was the case in \cref{ex_unfolding_cusp}.

If $F$ is weighted-homogeneous then, necessarily, $f$ is weighted-homogeneous.
Under some conditions for $(n,p)$ and the corank of $f$, when $f$ is weighted-homogeneous it always admits at least one weighted-homogeneous unfolding, hence substantial, see \cref{rem_good_weights_conditions}.
\end{example}

Notice in particular that any substantial vector field is projectable.
If $\eta$ is substantial then, dividing by the corresponding unit, for every $k=1,\ldots,m$ we can obtain projectable vector fields $\eta^k$ such that $d\Lambda_k(\eta^k) = \Lambda_k + \tilde q_k(X,\Lambda)$ with $\tilde  q_k(X,\Lambda)\in\mfr_\Lambda\mfr_{p+m}$.

Moreover, if  $\eta\in\Lift(F)$ is a projectable vector field and $\pi$ is the projection over the last $m$ coordinates, then the $1$-jet of $d\pi(\eta)$ only depends on $\Lambda_1,\ldots,\Lambda_m$, so it makes sense to consider the eigenvalues of $j^1d\pi(\eta)$ as an $m\times m$ matrix.
The following definition therefore makes sense:

\begin{definition}\label{def_weakly_subs}
We say that $F$ is {\it weakly substantial} if it admits a projectable $\eta\in\Lift(F)$ such that the $1$-jet of $d\pi(\eta)$ only has non-zero eigenvalues.
Such an $\eta$ is called a {\it weakly substantial vector field}.
\end{definition}

\begin{remark}\label{rem_substantial_nonzero_ev}
A substantial unfolding is also weakly substantial since a substantial vector field $\eta$ is projectable and the $1$-jet of $d\pi(\eta)$ has a matrix of the form
\begin{equation*}
\left(
\begin{matrix}
u_1(0,0) & 0 & \cdots & 0\\
0 & u_2(0,0) & \cdots & 0\\
\vdots & \vdots & \ddots & \vdots\\
0 & 0 & \cdots & u_m(0,0)
\end{matrix}
\right)
\end{equation*}
for some units $u_1,\ldots,u_m\in\ofu_{p+m}$.
\end{remark}

\begin{remark}
For the case in which $F$ is an OPSU, both notions of being substantial and weakly substantial are equivalent.

In fact, in this case both of them are equivalent to the original notion given in \cite{houstonaug2}, in which Houston only works with OPSUs and defines $F$ to be substantial if and only if $\Lambda\in d\Lambda(\Lift(F))$.
This obviously implies that $F$ is substantial in the sense of \cref{def_substantial}.
If $F$ is weakly substantial, then there exists $\eta\in\Lift(F)$ a projectable vector field such that $d\pi(\eta) = \eta_{p+1}$ has non-zero eigenvalues, with $\pi\colon(\bbk^p\times\bbk,0)\to(\bbk,0)$ the natural projection.
Since in this case $\eta_{p+1}$ is a single function, this means that $d\Lambda(\eta) = \eta_{p+1}(X,\Lambda) = \Lambda u_1(X,\Lambda)$ with $u_1\in\ofu_{p+1}$ a unit.
Therefore, after dividing by this unit, we get that $\Lambda\in d\Lambda(\Lift(F))$.

%
%

\end{remark}

\begin{definition}
Let $F(x,\lambda) = (f_\lambda(x),\lambda)$ and $G(x,\lambda) = (g_\lambda(x),\lambda)$ be two $m$-parameter unfoldings of $f$ and $g$.
$F$ and $G$ are {\it $\lambda$-equivalent} if there exist diffeomorphisms $\Psi(X,\Lambda) = (\psi_\Lambda(X),l(\Lambda))$ and $\Phi(x,\lambda) = (\phi_\lambda(x),l(\lambda))$ with $l$ a diffeomorphism such that
$$\Psi\circ F = G \circ \Phi.$$
If $\psi_0$, $\phi_0$ and $l$ are the identity mappings, then $F$ and $G$ are said to be equivalent as unfoldings.
\end{definition}

Notice in particular that if two unfoldings of $f$ and $g$ are $\lambda$-equivalent, then $f$ and $g$ are $\A$-equivalent, and if they are equivalent as unfoldings then $f=g$.

The notion of $\lambda$-equivalence is a particular instance of a more general equivalence relation, $\phi$-equivalence, which was first introduced in \cite{mendesfavaro1986} for the study of divergent diagrams and later used in \cite{manciniruas} to classify functions respecting some foliation.
In \cite{phiequiv}, this equivalence relation is used to study the simplicity of augmentations of singularities (see also \cite{augcod1morse}).



\begin{prop}
If $F$ and $F'$ are $m$-parameter stable unfoldings of $f$, and both are equivalent as unfoldings, if one of them is substantial then the other one is also substantial.
\end{prop}
\begin{proof}
Assume that $F$ is substantial.
Let $\Psi\colon(\bbk^p\times\bbk^m)\to(\bbk^p\times\bbk^m)$ and $\Phi\colon(\bbk^n\times\bbk^m)\to(\bbk^n\times\bbk^m)$ be of the form $\Psi(X,\Lambda) = (\psi_\Lambda(X),\Lambda)$ and $\Phi(x,\lambda) = (\phi_\lambda(x),\lambda)$ with $\psi_0(X) = X$ and $\phi_0(x)= x$ such that $\Psi\circ F = F'\circ \Phi$.
We will use $(X', \Lambda)$ for the variables in the target of $F'$.

Since $F$ is substantial, we can pick $\eta \in \Lift(F)$ of the form $d\Lambda_i(\eta) = \Lambda_i u_i + q_i(X,\Lambda)$ with $u_i\in \ofu_{p+m}$ a unit and $q_i \in \m_\Lambda\m_{p+m}$, so that using the isomorphism from \cref{lem_nishimura_lifts} we have
$$d\Lambda_i(d\Psi(\eta)\circ \Psi^{-1}) = \Lambda_i u_i(\psi_\Lambda^{-1}(X'),\Lambda) + q_i(\psi_\Lambda^{-1}(X'),\Lambda).$$
Therefore $F'$ is substantial.
\end{proof}

\begin{remark}
Notice that the same result holds if $\Psi$ and $\Phi$ are unfoldings of any diffeomorphism other than the identity.
\end{remark}

\begin{prop}\label{prop_weak_subs_preserved}
If $F$ and $F'$ are $\lambda$-equivalent $m$-parameter stable unfoldings of $f$ and $f'$, and one of them is weakly substantial, the other one is also weakly substantial.
\end{prop}

\begin{proof}
The argument is similar as in the proof of the last proposition, but we have that $\Psi(X,\Lambda) = (\psi_\Lambda(X),l(\Lambda))$ for some diffeomorphism $l\colon(\bbk^m,0)\to(\bbk^m,0)$.
Then, if $\pi$ is the projection over the last $m$ coordinates
$$d\pi\left( d\Psi(\eta)\circ\Psi^{-1}\right) = dl(d\pi(\eta))\circ \Psi^{-1}.$$
Since the $1$-jet of $d\pi(\eta)$ only depends on $\Lambda$ we have
$$j^1d\pi\left( d\Psi(\eta)\circ\Psi^{-1}\right) = dl(j^1(d\pi(\eta))\circ l^{-1})$$
therefore $d\pi\left(d\Psi(\eta)\circ\Psi^{-1}\right)$ has the same eigenvalues as $d\pi(\eta)$ by \cref{lemma_eigenvalues_preserve}.


\end{proof}

%

\begin{coro}\label{coro_subs_preserved_opsu}
If $F$ and $F'$ are $\lambda$-equivalent OPSUs and one of them is substantial, then the other one is also substantial.
\end{coro}
\begin{proof}
This is just \cref{prop_weak_subs_preserved} combined with the fact that for 1-parameter stable unfoldings both notions of substantiality coincide.
\end{proof}

\section{Necessary condition for weighted-homogeneity} \label{section_necessary}

Since $\mononp f$ is a $\K_e$-finite map-germ, we can find some vector fields $\bar f^1,\ldots,\bar f^m \in \theta(f)$ such that
\begin{equation}\label{eq_stable_generators}
\tke f + \linsp{\bbk}{\bar f^1(x),\ldots,\bar f^m(x)} + \linsp{\bbk}{\dpar{}{X_1},\ldots,\dpar{}{X_p}} = \theta(f)
\end{equation}
and $m$ being the minimal number necessary to satisfy this relation.
This is equivalent to saying that $\bar f^1,\ldots, \bar f^m$ form a basis as $\bbk$-vector space of the quotient of $\tke f + \tae f$ in $\theta(f)$.
Moreover, we can assume that $\bar{f}^k$ is a vector field whose components are all equal to zero except for one, which is a single monomial.
This is, for each $k= 1,\ldots,m$ there is some $j_k\in\lbrace 1,\ldots, p\rbrace$ such that we can write
\begin{equation*}
\bar{f}^k(x) = \sum_{j=1}^p \bar{f}^k_j(x) \dpar{}{X_j} = x^{\alpha_{j_k}}\dpar{}{X_{j_k}}
\end{equation*}
with $\bar f_j^k = 0$ for $j \neq j_k$ and $\bar f_{j_k}^k(x) = x^{\alpha_{j_k}}$ for some $\alpha_{j_k}\in  \bbn^n$.

\begin{example}\label{ex_h2}
Let $f\colon(\bbk^2,0)\to(\bbk^3,0)$ be given by $f(x,y) = (x,y^3,y^5+xy)$, then
\begin{equation*}
\tke f = \linsp{\ofu_2}{
\left(\begin{matrix}1\\0\\y\end{matrix}\right),
\left(\begin{matrix}0\\3y^2\\5y^4+x\end{matrix}\right)
} + \left\langle x, y^3\right\rangle \cdot \theta(f),
\end{equation*}
therefore
\begin{equation*}
\tke f + \linsp{\bbk}{\left(\begin{matrix}1\\0\\0\end{matrix}\right),\left(\begin{matrix}0\\1\\0\end{matrix}\right),\left(\begin{matrix}0\\0\\1\end{matrix}\right)}  +\linsp{\bbk}{\left(\begin{matrix}0\\y\\0\end{matrix}\right),\left(\begin{matrix}0\\0\\y^2\end{matrix}\right)} = \theta(f)
\end{equation*}
Here $\bar f^1(x,y) = y \dpar{}{X_2}, \bar f^2(x,y) = y^2\dpar{}{X_3}$
\end{example}

Although the results of this section will hold for any map-germ $f$ with finite $\K_e$-codimension, the particular case of map-germs with finite $\A_e$-codimension will be of special interest.
Keeping the notation above, we can prove that:

\begin{prop}\label{prop_normal_unfolding}
Every $m$-parameter stable unfolding of an $\A$-finite map-germ $f$ is $\lambda$-equivalent to one $F\colon(\bbk^n\times\bbk^m,0)\to(\bbk^p\times\bbk^m,0)$  of the form
\begin{equation}\label{eq_good_unfolding}
F(x,\lambda) = \left( f(x) + \sum_{k=1}^m \left(\lambda_k + \sum_{s=1}^{m}\lambda_s q_k^s(f(x),\lambda)\right)\bar f^k,\lambda\right)
\end{equation}
where $q_k^s(X,\Lambda)\in\ofu_{p+m}$ satisfies that $q_k^s(X,0)\in\mfr_p$. 
\end{prop}
\begin{proof}
Let $a = \aecod (f) < \infty$.
In \cite{notaopsus} it is proven that
$$\nae f = \linsp{\bbk}{\bar f^k(x)}_{k=1}^m + \linsp{\bbk}{\sum_{k=1}^m\tilde q_k^{l}(f(x))\bar f^k(x)}_{l=1}^{a-m}$$
with $\tilde q_k^l\in \mfr_{p}$ for $1\leq k\leq m$ and $1\leq l \leq a-m$.
Therefore, the unfolding $\mathcal F\colon (\bbk^n\times\bbk^a,0)\to(\bbk^p\times\bbk^a,0)$ given by
\begin{align*}
\mathcal F(x,\mu) &= \left(f(x) + \sum_{k=1}^m \mu_k \bar f^k(x) + \sum_{l=1}^{a-m} \mu_{m+l}\sum_{k=1}^m\tilde q_k^{l}(f(x))\bar f^k(x),\mu\right)\\
& = \left( f(x) + \sum_{k=1}^m \left(\mu_k + \sum_{l=1}^{a-m}\mu_{m+l}\tilde q_k^l(f(x))\right)\bar f^k(x),\mu\right)
\end{align*}
is a versal unfolding, which means that in particular every $m$-parameter unfolding of $f$ is equivalent as an unfolding (therefore $\lambda$-equivalent) to one of the form
$$\alpha^*\mathcal F(x,\lambda) = \left( f(x) + \sum_{k=1}^m \left(\alpha_k(\lambda) + \sum_{l=1}^{a-m}\alpha_{m+l}(\lambda)\tilde q_k^l(f(x))\right)\bar f^k(x),\lambda\right)$$
for some $\alpha\colon(\bbk^m,0)\to(\bbk^m\times\bbk^{a-m},0)$ given by $\alpha = (\alpha_1,\ldots,\alpha_{a})$.
In particular, for $\alpha^*\mathcal F$ to be stable necessarily it must satisfy the following (see Sections 4.2 and 4.3 and Lemma 5.5 in \cite{nunomond}):
\begin{multline*}
\theta(f) = \tke f + \linsp{\bbk}{\dpar{}{X_j}}_{j=1}^p + \\  \linsp{\bbk}{\sum_{k=1}^m \left(\dpar{\alpha_k}{\lambda_b}(0) + \sum_{l=1}^{a-m}\dpar{\alpha_{m+l}}{\lambda_b}(0)\tilde q_k^l(f(x))\right)\bar f^k(x)}_{b=1}^m
\end{multline*}

Since $\dpar{\alpha_l}{\lambda_b}(0)\tilde q_k^l(f(x))\bar f^k(x) \in \tke f$ for all $1\leq b,k\leq m$ and $1\leq l\leq a-m$, and $\bar f^k$ are picked such that they are minimal to satisfy \cref{eq_stable_generators}, then the matrix $(\dpar{\alpha_k}{\lambda_b}(0))_{k,b = 1}^m$ must be invertible, hence $\tilde \alpha \colon(\bbk^m,0)\to(\bbk^m,0)$ given by $\tilde \alpha = (\alpha_1,\ldots,\alpha_m)$ is a germ of a diffeomorphism, and so it has inverse $\tilde \alpha^{-1}$.

Define now $\Psi(X,\Lambda) = (X,\tilde\alpha(\Lambda))$ and $\Phi(x,\lambda) = (x,\tilde\alpha^{-1}(\lambda))$, which are diffeomorphisms of $\lambda$-equivalence, and let $\beta_l(\lambda) = \alpha_{m+l} \circ \tilde \alpha^{-1}(\lambda)$ for each $1\leq l \leq a-m$, then:

$$\Psi\circ \alpha^*\mathcal F\circ \Phi (x,\lambda) = \left( f(x) + \sum_{k=1}^m \left(\lambda_k + \sum_{l=1}^{a-m}\beta_{l}(\lambda)\tilde q_k^l(f(x))\right)\bar f^k(x),\lambda\right)$$
As $\beta_l(0)=0$, we can write $\beta_l(\lambda) = \sum_{s=1}^m \lambda_s \tilde\beta_l^s(\lambda)$ for some $\tilde\beta_l^s\in \ofu_m$, and then define $ q_k^s(X,\Lambda) = \sum_{l=1}^{a-m}\tilde\beta_l^s(\Lambda)\tilde q_k^l(X)$, which all satisfy $ q_k^s(X,0)=\sum_{l=1}^{a-m}\tilde\beta_l^s(0)\tilde q_k^l(X)\in\mfr_{p}$ and finally we have that $F$ is $\lambda$-equivalent to
$$\Psi\circ \alpha^*\mathcal F\circ \Phi (x,\lambda) = \left( f(x) + \sum_{k=1}^m \left(\lambda_k + \sum_{s=1}^{m}\lambda_s q_k^s(f(x),\lambda)\right)\bar f^k(x),\lambda\right)$$
as we required.
\end{proof}

The unfoldings in this form, although they are the most natural unfoldings to take in the practice, do not seem to be enough for studying all the possible cases regarding the relation of substantiality and weighted-homogeneity.
On one hand, non $\A$-finite map-germs might admit different stable unfoldings which are not $\lambda$-related to the ones in this class, but as we will see in this case these will be enough for our purposes.
On the other hand, $\lambda$-equivalence only preserves substantiality in the case of 1-parameter stable unfoldings.
But $\lambda$-equivalence does preserve weak substantiality in general, and this already becomes useful in the practical application of the following results.

\begin{example}
By \cref{prop_normal_unfolding} and using the computations from \cref{ex_h2}, all $2$-parameter stable unfoldings of $f(x,y) = (x,y^3, y^5+xy)$ are $\lambda$-equivalent to one of the form
$$\left(x, y^3 + \left(\lambda_1 + \sum_{s=1}^2\lambda_sq_1^s(f(x,y),\lambda)\right)y, y^5+xy + \left(\lambda_2 + \sum_{s=1}^2\lambda_sq_2^s(f(x,y),\lambda)\right)y^2,\lambda\right)$$ 
\end{example}


\begin{definition}
Assume $\mononp f$ is weighted-homogeneous with variable weights $w_1,\ldots,w_n$ and component degrees $d_1,\ldots, d_p$.
Then, using the previous notation, each $\bar f^k$ is weighted-homogeneous of some degree $d_{p+k}$.
We say that $f$ has {\it good weights} if for each $k=1,\ldots,m$ we have $d_{p+k}\neq d_{j_k}$.
If $f$ is quasi-homogeneous, we say that it has good weights if the weighted-homogeneous map that it is $\A$-equivalent to does.
\end{definition}

\begin{example}
Continuing with \cref{ex_h2}, we had that $f(x,y) = (x,y^3,y^5+xy)$ is a weighted-homogeneous map-germ with weights $w_1 = 4, w_2 = 1$ for $x$ and $y$ respectively, and degrees $d_1 = 4, d_2 = 3, d_3= 5$ for $X_1,X_2,X_3$ respectively.

We also had $\bar f^1 = y\dpar{}{X_2}$ and $\bar f^2 = y^2\dpar{}{X_3}$.
Here $j_1 = 2, j_2=3$ and $d_{4} = 1 < 3 = d_{j_1}, d_5 = 2 < 5 = d_{j_2}$, so $f$ has good weights.
In particular, this means that one of its stable unfoldings
$$F(x,y,\lambda) = (x,y^3+\lambda_1y,y^5+xy+\lambda_2y^2,\lambda)$$
is weighted-homogeneous by assigning weights $2$ and $3$ to $\lambda_1$ and $\lambda_2$ respectively.
Therefore, $F$ is substantial (recall \cref{ex_wh_one_substantial}).
The question now is: what other unfoldings are (weakly) substantial? We will see that all unfoldings of a weighted-homogeneous $\A$-finite map are either substantial or weakly substantial.
\end{example}

\begin{remark}\label{rem_good_weights_conditions}
In the case that $f$ admits an $m$-parameter stable unfolding with $(n+m,p+m)$ in the nice dimensions, or if $n=p$ and $f$ has corank 1, then if $f$ is quasi-homogeneous it has good weights. 
In particular, $d_{j_k} > d_{p+k}$ for all $1\leq k\leq m$, see Section 7.4 in \cite{nunomond}.
\end{remark}

\begin{theorem}\label{thm_wh_imp_subs}
If $f$ is a $\K_e$-finite weighted-homogeneous map-germ with good weights, then any stable unfolding $F$ in the form of \cref{eq_good_unfolding} is substantial. 
\end{theorem}
\begin{proof}
First we will obtain a candidate liftable vector field by replicating the proofs of Theorem 1 from \cite{nunoosetlifts} and Theorem 2 from \cite{nishimuralifts} in order to lift the Euler vector fields of $f$ to some $F$-related vector fields.
Then we will use the $\K_e$-finiteness of $f$ to check that this vector field satisfies the conditions we look for.

Using the previous notation, we define the Euler vector fields
\begin{align*}
\bar \eta(X) &= \sum_{j=1}^p d_jX_j\dpar{}{X_j}\\
\bar \xi(x) &= \sum_{i=1}^n w_ix_i\dpar{}{x_i}
\end{align*}
which satisfy $\bar\eta\circ f = df(\bar\xi)$.
Now, using the natural inclusions we can consider $\bar \eta \equiv (\bar \eta,0)$ and $\bar \xi \equiv (\bar \xi,0)$ as vector fields in $\theta_{p+m}$ and $\theta_{n+m}$ respectively by just adding zeroes.

We now want to compute the difference $\bar\eta(F) - dF(\bar \xi)$.
For simplicity, we will omit the variable $x$ when writing $f$ and $\bar f^k$ if there is no confusion.
Recall that $\bar f^k$ was picked so that it is a vector field which has all entries zero except for a single monomial in the component $1\leq j_k\leq p$.
On one hand we have that $dF(\bar\xi)$ is equal to
\begin{multline*}
\sum_{i=1}^n w_ix_i\left(\dpar{f}{x_i} + \sum_{k=1}^m\sum_{s=1}^m \lambda_s \dpar{q_k^s(f(x),\lambda)}{x_i}\bar f^k + \sum_{k=1}^m\left(\lambda_k+\sum_{s=1}^m\lambda_sq_k^s(f,\lambda)\right)\dpar{\bar f^k}{x_i} \right) \\
= \bar \eta (f) + \sum_{k,s=1}^m\lambda_s\bar f^k \sum_{i=1}^n x_iw_i \dpar{q_k^s(f(x),\lambda)}{x_i} +\sum_{k=1}^m\left(\lambda_k+\sum_{s=1}^m\lambda_sq_k^s(f,\lambda)\right)d_{p+k}\bar f^k
\end{multline*}
where the equality comes from the fact that $\bar \eta$ and $\bar \xi$ are $f$-related (since $f$ is weighted-homogeneous) and that $\bar f^k$ was picked so that its only monomial has weighted degree $d_{p+k}$. 
Now using the chain rule and again the fact that $f$ is weighted-homogeneous:
\begin{align*}
\sum_{i=1}^n x_iw_i \dpar{q_k^s(f(x),\lambda)}{x_i} &= \sum_{i=1}^n x_iw_i \sum_{l=1}^p\dpar{q_k^s}{X_l}(f,\lambda)\dpar{f_l}{x_i} \\
&= \sum_{l=1}^p \dpar{q_k^s}{X_l}(f,\lambda)\sum_{i=1}^n x_iw_i\dpar{f_l}{x_i} \\
&=  \sum_{l=1}^p \dpar{q_k^s}{X_l}(f,\lambda) d_lf_l.
\end{align*}
Hence, we have that $dF(\bar\xi)$ is equal to
\begin{equation*}
\bar \eta (f) + \sum_{k,s=1}^m\lambda_s\bar f^k \sum_{l=1}^p \dpar{q_k^s}{X_l}(f,\lambda) d_lf_l +\sum_{k=1}^m\left(\lambda_k+\sum_{s=1}^m\lambda_sq_k^s(f,\lambda)\right)d_{p+k}\bar f^k.
\end{equation*}
On the other hand, we have that $\bar \eta(F)$ is equal to
\begin{equation*}
\bar\eta(f) + \sum_{k=1}^m\left(\lambda_k +\sum_{s=1}^m \lambda_sq_k^s(f,\lambda)\right)d_{j_k}\bar f^k
\end{equation*}
using again the fact that $\bar f^k$ is just a single monomial in the $j_k$ component and zeroes in the rest of the entries.
Therefore, $\bar\eta(F) - dF(\bar\xi)$ is equal to
\begin{multline*}
\sum_{k=1}^m\left(\lambda_k + \sum_{s=1}^m\lambda_sq_k^s(f,\lambda)\right)(d_{j_k}-d_{p+k})\bar f^k - \sum_{k,s=1}^m\lambda_s\bar f^k\sum_{l=1}^p\dpar{q_k^s}{X_l}(f,\lambda)d_lf_l\\
= \sum_{k=1}^m(d_{j_k}-d_{p+k})\lambda_k\bar f^k + \sum_{k,s=1}^m\left((d_{j_k}-d_{p+k})q_k^s(f,\lambda) - \sum_{l=1}^p\dpar{q_k^s}{X_l}(f,\lambda)d_lf_l\right)\lambda_s\bar f^k
\end{multline*}
In the last term, we can switch the indexes $s$ and $k$ to obtain that $\bar\eta(F) - dF(\bar\xi)$ is equal to
\begin{equation*}
\sum_{k=1}^m\lambda_k\left((d_{j_k}-d_{p+k})\bar f^k + \sum_{s=1}^m\left((d_{j_s}-d_{p+s})q_s^k(f,\lambda) - \sum_{l=1}^p\dpar{q_s^k}{X_l}(f,\lambda)d_lf_l\right)\bar f^s\right)
\end{equation*}

Now, for each $k = 1,\ldots,m$ write
\begin{align*}
\tau^k(x,\lambda) & =  \sum_{s=1}^m\left((d_{j_s}-d_{p+s})q_s^k(f,\lambda) - \sum_{l=1}^p\dpar{q_s^k}{X_l}(f,\lambda)d_lf_l\right)\bar f^s\\
\gamma^k(x,\lambda) & = (d_{j_k}-d_{p+k})\bar f^k 
\end{align*}
so that $\bar\eta(F) - tF(\bar\xi) = \sum_{k=1}^m\lambda_k(\gamma^k + \tau^k)$.
Notice that $\tau^k(x,0) \in f^*\m_{p}\theta(f) \subseteq \tke f$.

Since $F$ is a stable unfolding, $\gamma^k + \tau^k\in\tae F$ and so there exist some $\xi^k = \sum_{i=1}^{n}\xi_i^k\dpar{}{x_i}+\sum_{i=1}^m\xi_{n+i}^k\dpar{}{\lambda_i}\in\theta_{n+m}$ and $\eta^k = \sum_{j=1}^p \eta_j^k\dpar{}{X_j}+\sum_{j=1}^m \eta_{p+j}^k\dpar{}{\Lambda_j}\in\theta_{p+m}$ such that
\begin{equation}\label{eq_gammak}
\gamma^k + \tau^k = tF(\xi^k) + \eta^k(F).
\end{equation}
From this we can already obtain our candidate for a substantial, liftable vector field, since
$$\bar\eta(F) - tF(\bar\xi) = \sum_{k=1}^m\lambda_k(tF(\xi^k)+\eta^k(F))$$
and therefore
$$(\bar\eta - \sum_{k=1}^m\Lambda_k \eta^k)\circ F = tF(\sum_{k=1}^m\lambda_k\xi^k - \bar \xi)$$
Call $\eta = \bar\eta - \sum_{k=1}^m\Lambda_k \eta^k$ and $\xi = \sum_{k=1}^m\lambda_k\xi^k - \bar \xi$, then what we just showed is that $\eta$ and $\xi$ are $F$-related.
Now we have $d\Lambda_{k_0}(\eta) = \sum_{k=1}^m\Lambda_k\eta_{p+k}^{k_0}$, so in order to prove that $\eta$ is a substantial vector field we only need to check for every $1\leq {k_0} \leq m$ that $\eta_{p+k_0}^{k_0}$ is a unit (i.e., that $\eta_{p+k_0}^{k_0}(0)\neq 0$)  and $\eta_{p+k}^{k_0}(0) = 0$ for each $k\neq k_0$. 

We will do this by reaching a contradiction with the fact that the $\bar f^1,\ldots,\bar f^m$ are picked as in \cref{eq_stable_generators}.
Fix $k_0 \in\lbrace 1,\ldots,m\rbrace$.
First, looking at the last $m$ components in \cref{eq_gammak}, we have that for each $k=1,\ldots,m$
\begin{equation}\label{eq_xi_equal_eta}
\xi_{n+k}^{k_0}(x,\lambda) = \eta_{p+k}^{k_0}\circ F.
\end{equation}
Let's expand the right-hand side of \cref{eq_gammak}.
Again, we will omit $x$ and $\lambda$ from $f,\bar f^k$ and $\xi^{k_0}$ when convenient if there is no confusion.
On one side, $dF(\xi^{k_0})$ is equal to
\begin{multline}\label{eq_df_xik}
\sum_{i=1}^n\dpar{f}{x_i}\xi_i^{k_0} + \sum_{i=1}^n\xi^{k_0}_i\sum_{v=1}^m\left(\sum_{s=1} ^m\lambda_s \bar f^v\sum_{l=1}^p\dpar{q_v^s}{X_l}(f,\lambda)\dpar{f_l}{x_i} + \left(\lambda_v + \sum_{s=1}^m\lambda_sq_s^v(f,\lambda)\right)\dpar{\bar f^v}{x_i}\right)\\
+ \sum_{i=1}^m\sum_{v=1}^m\dpar{\lambda_v+\sum_{s=1}^m\lambda_sq_v^s(f,\lambda)}{\lambda_i}\bar f^v\xi^{k_0}_{n+i}+ \sum_{v=1} ^m \xi^{k_0}_{n+v} \dpar{}{\Lambda_{v}}
\end{multline}
Notice that here we are mixing vector field notations for convenience: $f$ and $\bar f^1,\ldots,\bar f^m$ act as vector fields in $\ofu_{n}^{p+m}$ whose last $m$ components (corresponding to $\dpar{}{\Lambda_1},\ldots,\dpar{}{\Lambda_m}$) are all zero.
We expand the first term in the second line of the last equation and use \cref{eq_xi_equal_eta} to obtain
\begin{multline*}\label{eq_derivative_expansion}
\sum_{i,v=1}^m\dpar{(\lambda_v+\sum_{s=1}^m\lambda_sq_v^s(f,\lambda))}{\lambda_i}\bar f^v\xi^{k_0}_{n+i} = 
\sum_{v=1}^m\eta^{k_0}_{n+v}(F)\bar f^v + \sum_{i,v,s=1}^m\dpar{(\lambda_sq_v^s(f(x),\lambda))}{\lambda_i}\xi^{k_0}_{n+i}\bar f^v
\end{multline*}

Write $\rho(x,\lambda) = \sum_{i,v,s=1}^m\dpar{\lambda_sq_v^s(f(x),\lambda)}{\lambda_i}\xi^{k_0}_{n+i}\bar f^v$.
Finally, we take a look at the first $p$ components in \cref{eq_gammak} and take $\lambda = 0$ so that, using \cref{eq_xi_equal_eta} and \cref{eq_df_xik}, we obtain
\begin{multline*}
(d_{j_{k_0}}-d_{p+k_0})\bar f^{k_0} + \tau^{k_0}(x,0) = \\ \sum_{i=1}^n\dpar{f}{x_i}\xi_i^{k_0}(x,0) + \sum_{v=1}^m \eta_{p+v}^{k_0}(f(x),0)\bar f^v +  \rho(x,0) + \sum_{j=1}^p \eta^{k_0}(f(x),0)
\end{multline*}
which is equivalent to
\begin{multline*}
(d_{j_{k_0}}-d_{p+k_0})\bar f^{k_0} -\sum_{v=1}^m \eta^{k_0}_{p+v}(0)\bar f^v = \\ \sum_{i=1}^n\dpar{f}{x_i}\xi_i^{k_0}(x,0) + \sum_{v=1}^m (\eta_{p+v}^{k_0}(f(x),0) -\eta^{k_0}_{p+v}(0))\bar f^v +  \rho(x,0) + \sum_{j=1}^p \eta^{k_0}(f(x),0) - \tau^{k_0}(x,0)
\end{multline*}
The left-hand side of this equation is a linear combination in $\bbk$ of the $\bar f^1,\ldots,\bar f^m$.
If we prove that the right-hand side is an element of $\tke f + \linsp{\bbk}{\dpar{}{X_1},\ldots,\dpar{}{X_p}}$, then since $m$ was picked to be minimal as to satisfy \cref{eq_stable_generators}, necessarily all the coefficients in the left-hand side will be zero. 
In particular this will mean that $\eta_{p+k_0}^{k_0}(0) = d_{j_{k_0}}-d_{p+k_0} \neq 0$, since $f$ has good weights by hypothesis, and that for all $v\neq k_0$ we have $\eta_{p+v}^{k_0}(0) =0$, as we wanted to prove.

Now, to see that the right-hand side belongs to $\tke f$ we look at its individual summands: the first term belongs to $tf(\theta_n)$, the second and fourth belong to $f^\star\mfr_p\theta(f) +\linsp{\bbc}{\dpar{}{X_1},\ldots,\dpar{}{X_p}}$, and we already saw that $\tau(x,0)\in f^\star\mfr_p\theta(f)$.
Therefore it only remains to see that $\rho(x,0) \in \tke f + \linsp{\bbc}{\dpar{}{X_1},\ldots,\dpar{}{X_p}}$.
We have that $\rho(x,\lambda)$ is equal to
\begin{multline*}\sum_{i,v,s=1} \dpar{(\lambda_s q_v^s(f(x),\lambda)}{\lambda_i} \xi_{n+i}^{k_0} \bar f^v = \sum_{v=1}^m  \sum_{i=1}^m\sum_{s=1}^m\left(\dpar{\lambda_s}{\lambda_i} q_v^s(f(x),\lambda) + \lambda_s\dpar{q_v^s(f(x),\lambda)}{\lambda_i}\right)\xi_{n+i}^{k_0} \bar f^v
\end{multline*}
therefore
$$\rho(x,0) = \sum_{v=1}^m \sum_{i=1}^m\sum_{\substack{s=1\\s\neq i}}^m q_v^s(f(x),0)\xi_{n+i}^{k_0}\bar f^v$$
which belongs to $f^*\mfr_p\theta(f)\subseteq \tke f$ since $q_v^s(X,0)\in\mfr_p$.

\end{proof}






\begin{coro}\label{coro_wh_imp_weakly_subs}
Every stable unfolding with minimal number of parameters of an $\A$-finite, weighted-homogeneous map-germ with good weights is weakly substantial.
\end{coro}

\begin{proof}
By \cref{prop_normal_unfolding}, every stable unfolding with minimal number of parameters is $\lambda$-equivalent to one in the form of \cref{eq_good_unfolding}, and by \cref{thm_wh_imp_subs} all of these are substantial, therefore weakly substantial.
Now by \cref{prop_weak_subs_preserved}, weak substantiality is preserved by $\lambda$-equivalence and so the result follows.
\end{proof}

\begin{coro}\label{coro_every_unf_weakly_subs}
If $\mononp f$ is $\A$-finite and quasi-homogeneous with good weights, then every stable unfolding with minimal number of parameters of $f$ is weakly substantial.
\end{coro}
\begin{proof}
Since $f$ is $\A$-equivalent to a weighted-homogeneous $\mononp g$, let $\psi\colon(\bbk^p,0)\to(\bbk^p,0)$ and $\phi\colon(\bbk^n,0)\to(\bbk^n,0)$ be germs of diffeomorphism such that $\psi\circ f \circ \phi = g$.
Let $F$ be a stable unfolding of $f$, and define $\Psi(X,\Lambda) = (\psi(X),\Lambda)$ and $\Phi(x,\lambda) = (\phi(x),\lambda)$.
Then $G = \Psi\circ F\circ \Phi$ is an unfolding of $g$ and is therefore weakly substantial by \cref{coro_wh_imp_weakly_subs}.
Since $G$ is $\lambda$-equivalent to $F$, the result follows by \cref{prop_weak_subs_preserved}.
\end{proof}

\begin{coro}\label{coro_every_opsu_subs}
If $\mononp f$ is $\A$-finite and quasi-homogeneous with good weights and admits an OPSU, then every OPSU of $f$ is substantial.
\end{coro}
\begin{proof}
Using \cref{coro_every_unf_weakly_subs}, all unfoldings of $f$ are weakly substantial.
By \cref{coro_subs_preserved_opsu}, in this case all weakly substantial unfoldings are substantial, hence the result follows.
\end{proof}

As an application, we can check when a map-germ is not quasi-homogeneous:

\begin{example}
Let $f\colon(\bbc^5,0)\to(\bbc^5,0)$ be the map-germ given by
$$f(x,y,u) = (x^3 +y^3 + u_1x+u_2y-(a u_3+u_3^2)x^2+u_3y^2,xy,u)$$
for $a\neq 0,-1$. 
This appears in \cite{extranicedim} as an example of a corank 2, non-simple map-germ of $\A_e$-codimension 2.
It admits the 1-parameter stable unfolding
$$F(x,y,u,v) = (x^3 +y^3 + u_1x+u_2y-(au_3+u_3^2-v)x^2+u_3y^2,xy,u,v).$$
It is too computationally expensive to obtain directly the liftable vector fields of $F$, but it is easy to see that if $\phi(x,y,u,v) = (x,y,u,v+au_3+u_3^2)$ and $\psi(X,Y,U,V)  = (X,Y,U,V-aU_3-U_3^2)$ then
$$\psi\circ F \circ \phi(x,y,u,v) = (x^3 +y^3 + u_1x+u_2y+vx^2+u_3y^2,xy,u,v)$$
whose liftable vector fields are computed also in \cite{extranicedim}.
Using the isomorphism from \cref{lem_nishimura_lifts}, it is easy to check that $F$ is not substantial, therefore $f$ cannot be quasi-homogeneous with good weights.
\end{example}

\begin{example}
Let $f\colon(\bbc^3,0)\to(\bbc^4,0)$ be the map-germ given by
$$f(x,y,z) = (x,y,yz+z^6+z^8,xz+z^3)$$
which is the $P_3^2$ singularity as it appears in the classification \cite{houstonkirk}.
In the paper, the authors compute that:
$$\mu_I(f) = 4 > 3 = \aecod (f)$$
But since the Mond conjecture has not been proven in the dimension pair $(3,4)$, this is not enough to prove that $f$ is not quasi-homogeneous.
In order to prove this, we can use SINGULAR to easily compute the liftable vector fields of its OPSU:
$$F(x,y,z) = (\lambda,x,y,yz+\lambda z^2+z^6+z^8,xz+z^3)$$
and check that it is not substantial, hence it is not quasi-homogeneous with good weights.
In particular, since $F$ is in the nice dimensions, $f$ cannot be quasi-homogeneous.
\end{example}

\section{Equidimensional, corank 1, $\A$-finite case} \label{section_converse}

In this section we want to prove that, in the analytic, equidimensional, corank 1, $\A$-finite case, if the analytic stratum of the stable unfolding has dimension 0 or if the map-germ has multiplicity 3, then admitting a substantial unfolding is equivalent to being quasi-homogeneous.

\subsection{Minimal stable unfolding case} \label{subsection_minimal}

Here we restrict ourselves to the complex case.
First we assume that $f\colon(\bbc^n,0)\to(\bbc^n,0)$ admits a minimal stable unfolding, that is, it has analytic stratum of dimension zero (recall \cref{def_frelated}).



The goal now is, assuming $f$ admits a minimal substantial unfolding, to find a coordinate system in which there exist $f$-related Euler vector fields whose eigenvalues are all positive integers, since this is equivalent to being weighted-homogeneous (recall the discussion after \cref{def_weighted_homogeneous}).


First we will study the possible eigenvalues of minimal stable singularities.
The only $\A$-classes of stable singularities in this setting are those given by the stable unfoldings with minimal number of parameters of $A_n$ singularities, which are minimal, and prisms of these classes which have analytic stratum of positive dimension.

We check that, after multiplication by a constant, the eigenvalues of a liftable vector field (and its corresponding lowerable vector field) of a minimal stable singularity will be either all zeroes or all positive integer values.
Hence, a substantial vector field of a minimal stable unfolding will be a projectable vector field with positive, rational eigenvalues (positive integers after multiplication by constant).

Since substantial vector fields are projectable, this will mean that if a corank 1 map-germ $f\colon(\bbc^n,0)\to(\bbc^n,0)$ admits a substantial stable unfolding, in particular it will admit $f$-related vector fields $\eta$ and $\xi$ in source and target with positive integer eigenvalues.
Using their Poincaré-Dulac normal forms, the result will follow.


%
%

Consider the $A_n$ singularity $f(x_1) = x_1^{n+1}$, it admits the following stable unfolding:
$$F(x_1,\ldots,x_n) = (x_1^{n+1} + x_2x_1 + \cdots + x_nx_1^{n-1},x_2,\ldots,x_n).$$
The linear parts of the generators of $\Lift(F)$ are computed in several places, for instance in \cite{arnold_wavefronts, bruce_envelopes} amongst others.
From here we get that there exists an $\ofu_n$-basis of $\Lift(F)$ given by $\eta_1,\ldots,\eta_n$, which can be decomposed as $\eta_j = \eta_j^L + \eta_j^{\geq 2}$ with $\eta_j^{\geq 2}$ a vector field with terms only of degree 2 or higher for each $j=1,\ldots,n$, and $\eta_j^L$ satisfying
\begin{equation}\label{eq_ak_lift_1jet}
\begin{split}
\eta_1^L(X) & = (n+1)X_1\dpar{}{X_1} + nX_2\dpar{}{X_2} + \cdots + 2X_n\dpar{}{X_n}\\
\eta_2^L(X) & = (n+1)X_1\dpar{}{X_2} + nX_2\dpar{}{X_3} + \cdots + 3X_{n-1}\dpar{}{X_{n}}\\
& \; \; \vdots \\
\eta_{n-1}^L(X) & = (n+1)X_1\dpar{}{X_{n-1}}+nX_2\dpar{}{X_n}\\
\eta_n^L(X) & = (n+1)X_1\dpar{}{X_n}
\end{split}
\end{equation}

\begin{lemma}\label{lemma_vp_lift}
Any $\eta\in\Lift(F)$ is of the form $\eta = \sum_{j=1}^n b_j(X)\eta_j(X)$, and its eigenvalues are given by the vector $b_1(0)\cdot (n+1,n,\ldots,2)$.
\end{lemma}
\begin{proof}
This comes directly since the $1$-jet of $\eta$ is given by the lower-triangular matrix
\begin{equation}
\left(\begin{matrix}
b_1(0)(n+1) & 0 & \cdots & 0\\
b_2(0)(n+1) & b_1(0)n & \cdots & 0 \\
\vdots &\vdots &\ddots & \vdots\\
b_n(0)(n+1) & b_{n-1}(0)n & \cdots & b_1(0)2
\end{matrix}\right)
\end{equation}
\end{proof}

\begin{lemma}\label{lemma_vp_low}
If $\eta\in\Lift(F)$ is of the form $\eta = \sum_{j=1}^n b_j(X)\eta_j(X)$, then it admits an $F$-related lowerable vector field $\xi\in\Low(F)$ with eigenvalues given by the vector $b_1(0)\cdot(1,n,n-1,\ldots,2)$.
\end{lemma}
\begin{proof}
First, we need to study the linear parts of the lowerable vector fields associated to each $\eta_1,\ldots,\eta_n$.
Let $\xi_i\in\Low(F)$ be $F$-related to $\eta_i$ for each $i=1,\ldots,n$.
Since $\eta_1$ is the Euler vector field, we can consider:
$$\xi_1 = x_1\dpar{}{x_1}+ nx_2\dpar{}{x_2} + \cdots + 2x_n\dpar{}{x_n}$$
Fix any $2\leq i \leq n$, we can write $\xi_i = \xi_i^L + \xi_i^{\geq 2}$ where $\xi_i^L$ is the $1$-jet of $\xi_i$ and $\xi_i^{\geq 2}$ has only terms of degree 2 or higher.
Then
%
\begin{equation*}
tF(\xi_i) = \left( \begin{matrix}
\left((n+1)x_1^n+x_2+\cdots +(n-1)x_nx_1^{n-2}\right)(\xi_i)_1 + x_1(\xi_i)_2 + \cdots + x_1^{n-1}(\xi_i)_n\\
(\xi_i^L)_2 + (\xi_i^{\geq 2})_2\\
\vdots \\
(\xi_i^L)_n + (\xi_i^{\geq 2})_n
\end{matrix}\right)
\end{equation*}
Here $(\cdot)_j$ denotes the $j$-component of the corresponding vector field.
We also have
\begin{equation*}
\eta_i\circ F = \left(\begin{matrix}
(\eta_i^{\geq 2})_1\circ F\\
\vdots \\
(\eta_i^{\geq 2})_{i-1} \circ F \\
(n+1)(x_1^{n+1} +x_2x_1+\cdots + x_nx_1^{n-1}) + (\eta_i^{\geq 2})_i\circ F\\
nx_2 + (\eta_i^{\geq 2})_{i+1}\circ F \\
\vdots\\
(i+1)x_{n+1-i} + (\eta_i^{\geq 2})_n\circ F
\end{matrix}\right)
\end{equation*}
Since $tF(\xi_i) = \eta_i\circ F$, we obtain
\begin{align*} 
(\xi_i^L)_{j} &= \left\lbrace\begin{matrix*}[c]
0 & \text{ if } 1 < j \leq i \\
(n+i+1-j)x_{j-i+1} & \text{ if } i < j \leq n
\end{matrix*}\right. \\
(\xi_i^{\geq 2})_{j} &= \left\lbrace\begin{matrix*}[c]
(\eta_i^{\geq 2}\circ F)_j & \text{ if } j\neq i\\
(n+1)(x_1^{n+1}+x_2x_1+\cdots+x_nx_1^{n-1}) + (\eta_i^{\geq 2})_i\circ F & \text{ if }j=i
\end{matrix*}\right.
\end{align*} 

Now we only need to compute $\xi_1^L$, which we can express as $\xi_1^L = \sum_{j=1}^n c_{i,j}x_j$ for some $c_{i,1},\ldots,c_{i,n}\in\bbc$.
In fact, for our purposes, we only need to compute $c_{i,1}$.
Looking now at the first component of equation $tF(\xi_i) = \eta_i\circ F$ we obtain the following:
\begin{multline*}
\left((n+1)x_1^n+x_2+\cdots +(n-1)x_nx_1^{n-2}\right)(\sum_{j=1}^n c_{i,j}x_j + (\xi_i^{\geq 2})_1) \\+ x_1(\eta_i^{\geq 2})_2\circ F + \cdots + x_1^{i-2}(\eta_i^{\geq 2})_{i-1}\circ F 
 \\+ x_1^{i-1}\left((n+1)(x_1^{n+1} + x_2x_1 + \cdots +x_nx_1^{n-1})+(\eta_i^{\geq 2})_i\circ F\right) \\
 + x_1^i(nx_2+(\eta_i^{\geq 2})_{i+1}\circ F) + \cdots + x_1^{n-1}((i+1)x_{n+1-i} + (\eta_i^{\geq 2})_n \circ F) \\
 = (\eta_i^{\geq 2})_1\circ F 
\end{multline*}
Searching for the coefficient of $x_1^{n+1}$ in this equation, we obtain that $c_{i,1} = 0$.
To summarize, the linear parts of each $\xi_i$ can be expressed as
\begin{align*}
\xi_1^L & = x_1\dpar{}{x_1} + nx_2\dpar{}{x_2} + \cdots + 2x_n\dpar{}{x_n} \\
\xi_2^L & = (\sum_{j=2}^n c_{2,j}x_j)\dpar{}{x_1} + nx_2\dpar{}{x_3} + \cdots + 3x_{n-1}\dpar{}{x_n}\\
& \; \; \vdots\\
\xi_{n-1}^L & = (\sum_{j=2}^n c_{n-1,j}x_j)\dpar{}{x_1}  + nx_2\dpar{}{x_n}\\
\xi_n^L & = (\sum_{j=2}^n c_{n,j}x_j)\dpar{}{x_1} 
\end{align*}
for some $c_{i,j}\in\bbc$.
Now, recall that we had $\eta = \sum_{j=1}^n b_j\eta_j$, hence it is $F$-related to the lowerable vector field $\xi(x) = \sum_{j=1}^n b_j(F(x))\xi_j(x)$.
Finally, we can compute the $1$-jet of $\xi$, which is given by the following matrix:
\begin{equation*}
\left(\begin{matrix}
b_1(0) & \sum_{j=2}^nb_j(0)c_{j,2} &  \sum_{j=2}^nb_j(0)c_{j,3} & \cdots & \sum_{j=2}^nb_j(0)c_{j,n-1} &  \sum_{j=2}^nb_j(0)c_{j,n}\\
0 & b_1(0)n & 0 & \cdots & 0 & 0 \\
0 & b_2(0)n & b_1(0)(n-1) & \cdots & 0 & 0\\
\vdots & \vdots & \vdots & \ddots & \vdots & \vdots\\
0 & b_{n-2}(0)n & b_{n-3}(0)(n-1) & \cdots & b_1(0)3 & 0\\
0 & b_{n-1}(0)n & b_{n-2}(0)(n-1) & \cdots & b_2(0)3 & b_1(0)2
\end{matrix}\right)
\end{equation*}
The result follows now by realizing that the eigenvalues of this matrix are those of the entries in the diagonal.
\end{proof}

We can finally prove:


\begin{theorem}\label{thm_minimal}
Let $f\colon(\bbc^{n},0)\to(\bbc^{n},0)$ be a corank 1, $\A$-finite map-germ with minimal stable unfolding. 
Then, the following statements are equivalent:
\begin{enumerate}
\item $f$ is quasi-homogeneous.\label{item_minimal_f_wh}
\item $f$ admits a substantial stable unfolding.\label{item_minimal_exists_subs}
\item Every unfolding of $f$ in the form of \cref{eq_good_unfolding} is substantial.\label{item_minimal_good_form_subst}
\item Every stable unfolding of $f$ is weakly substantial.\label{item_minimal_all_weak}
\end{enumerate}
\end{theorem}
\begin{proof}
The fact that \ref{item_minimal_f_wh} implies the rest of items is by \cref{thm_wh_imp_subs} and \cref{coro_every_unf_weakly_subs}.
That \ref{item_minimal_good_form_subst} implies \ref{item_minimal_exists_subs} is direct.
We will prove that \ref{item_minimal_exists_subs} implies \ref{item_minimal_f_wh} and \ref{item_minimal_all_weak} implies \ref{item_minimal_f_wh} at the same time, since the arguments are analogous.

Let $F\colon(\bbc^n\times\bbc^m,0)\to(\bbc^n\times\bbc^m,0)$ be a minimal stable substantial or weakly substantial unfolding. Then it is $\A$-equivalent to the stable unfolding with minimal number of parameters of the $A_{n+m}$-singularity.

Let $\eta$ be a substantial (respectively, weakly substantial) vector field, then if $\pi\colon(\bbc^n\times\bbc^m,0)\to(\bbc^m,0)$ is the natural projection, the $1$-jet $d\pi(\eta)$ can be seen as an $m\times m$ matrix whose eigenvalues are all non-zero (recall \cref{def_weakly_subs} and \cref{rem_substantial_nonzero_ev}).

In particular, $\eta$ has some non-zero eigenvalue.
By \cref{lemma_vp_lift} the vector of all eigenvalues of $\eta$ must be a multiple of $(n+1,n,\ldots,2)$ by a single constant, and since one eigenvalue is non-zero, this constant must also be non-zero.
Therefore, after dividing by this constant we can assume that $\eta$ has all positive integer eigenvalues.


Define $\eta_0(X) = \pi_1(\eta)(X,0)$ with $\pi_1$ the projection over the first $n$ components and let $d_1,\ldots,d_n > 0$ be the integer eigenvalues of $\eta_0$ (recall \cref{rem_eigenvalues_projection_subset}) .
Then, by \cref{prop_lifts_unfolding} we have that $\eta_0\in\Lift(f)$.

Let now $H\colon(\bbc^{n}\times\bbc^m,0)\to(\bbc, 0)$ be a reduced equation of $\Delta F$.
In particular, $h(X) = H(X,0)$ is a reduced equation of $\Delta f$.
Since we are in the equidimensional case, by \cref{prop_lift_eq_derlog} we have $\Lift(F) = \Derlog(\Delta F)$, and $\Lift(f) = \Derlog(\Delta f)$.

This means $\eta_0(H) = a\cdot H$ for some $a\in\ofu_{n}$.
By \cref{thm_poincare_dulac}, there is some diffeomorphism $\psi\colon(\bbc^n,0)\to(\bbc^n,0)$ such that $d \psi \circ \eta_0\circ\psi^{-1} = \eta_S + \eta_N$, with $\eta_S = \sum_{j=1}^n d_jX_j\dpar{}{X_j}\in\theta_n$ and some $\eta_N\in\theta_n$ with nilpotent $1$-jet.
The same theorem ensures that $\eta_S(h\circ \psi^{-1}) = (a\cdot h)\circ\psi^{-1})$.
Now, since $h\circ\psi^{-1}$ is a reduced equation of the discriminant of $\psi\circ f$, in particular this means $\eta_S\in\Lift(\psi\circ f)$.

We will proceed working in the coordinates induced by $\psi$ to simplify notation, so that $\eta_S \in\Lift(f)$

Let $\xi\in\theta_n$ be an $f$-related vector field to $\eta_S$.
On one hand, the eigenvalues of $\xi$ must be all positive: this is due to the fact that we can lift $\xi,\eta_S$ to a pair of $F$-related projectable vector fields $\tilde \xi$ and $\tilde \eta_S$ by \cref{prop_lifts_unfolding}.
Since $\tilde{\eta_S}$ is projectable, it has $d_1,\ldots,d_n$ among its eigenvalues, so by \cref{lemma_vp_lift} and \cref{lemma_vp_low} all of the eigenvalues of $\xi$ are positive integers, call them $w_1,\ldots,w_n$.

On the other hand, being $f$-related means that $df(\xi) = \eta_S\circ f$, in particular for each $j=1,\ldots,n$
$$ \xi(f_j) = \sum_{i=1}^n \xi_i \dpar{f_j}{x_i} =  \eta_{S,j}\circ f = d_j\cdot f_j $$
where $\eta_{S,j}(X) = d_jX_j$ is the $j$-component of $\eta_S$.

Applying now \cref{thm_poincare_dulac}, there exists some $\phi\colon(\bbc^n,0)\to(\bbc^n,0)$ diffeomorphism such that $d\phi\circ \xi\circ \phi^{-1} = \xi_S + \xi_N$ with $\xi_S = \sum_{i=1}^n w_ix_i\dpar{}{x_i}\in\theta_n$ and $\xi_N\in\theta_N$ with nilpotent $1$-jet such that, in the new coordinates induced by $\phi$ we have
$$\xi_S(f_j) = d_j \cdot f_j = \eta_{S,j}\circ f$$  
that is, $df(\xi_S) = \eta_S\circ f$, which implies that $f$ is weighted-homogeneous in the coordinates induced by $\psi,\phi$.
\end{proof}

\subsection{Multiplicity 3 case} \label{section_mult3}

Consider the minimal stable unfolding of the $A_2$ singularity:
$$G(y,\lambda) = (y^3+\lambda y,\lambda).$$
An easy computation shows
\begin{equation}\label{eq_gens_a3}
\Lift(G) = \linsp{\ofu_2}{
3Y\dpar{}{Y} + 2\Lambda\dpar{}{\Lambda},-\frac{2}{3}\Lambda^2\dpar{}{Y} + 3Y\dpar{}{\Lambda}
}.
\end{equation}



Let $f\colon(\bbk^n,0)\to(\bbk^n,0)$ be a corank 1, $\A$-finite, multiplicity 3 map-germ.
Then we can assume $f$ is ($\A$-equivalent to) a map-germ of the form
$$f(x,y) = (x,y^3+ q(x)y)$$
where $x\in\bbk^{n-1}$, $y\in\bbk$ and $q\in\ofu_{n-1}$.
We can assume moreover that $q\in \mfr_{n}^2$ and that it is singular, or else we are dealing with either a regular map or a stable map and in both cases it is already equivalent to a weighted-homogeneous map-germ.

In this case, $f$ always admits the stable 1-parameter unfolding
\begin{equation}\label{eq_mult_3_opsu}
F(x,y,\lambda) = (x,y^3 +(q(x)+\lambda)y,\lambda)
\end{equation}

\begin{theorem}\label{thm_mult3}
Let $f\colon(\bbk^n,0)\to(\bbk^n,0)$ be a corank 1, $\A$-finite, non-stable, multiplicity 3 map-germ, the following are equivalent:
\begin{enumerate}
\item $f$ is quasi-homogeneous. \label{item_mult3_qh}
\item The unfolding in \cref{eq_mult_3_opsu} is substantial. \label{item_mult3_one_subs}
\item Every 1-parameter stable unfolding of $f$ is substantial. \label{item_mult3_every_subs}
\end{enumerate}

\end{theorem}

\begin{proof}
By \cref{coro_every_opsu_subs} we only need to prove that \ref{item_mult3_one_subs} implies \ref{item_mult3_qh}.
Consider the germs of diffeomorphism $\Phi,\Psi\colon(\bbk^{n-1}\times\bbk\times\bbk,0)\to(\bbk^{n-1}\times\bbk\times\bbk,0)$ given by
\begin{align*}
\Phi(x,y,\lambda) &= (x,y,\lambda-q(x))\\
\Psi(X,Y,\Lambda) &= (X,Y,\Lambda+q(X))
\end{align*}
then $\Psi\circ F \circ \Phi (x,y,\lambda) = (x, y^3+\lambda y,\lambda) = (\Id_{n-1}\times G)(x,y,\lambda)$.
Now \cref{lem_nishimura_lifts} ensures that:
\begin{align*}
\Lift(\Id_{n-1}\times G) &\to \Lift(F)\\
\eta & \mapsto d\Psi^{-1}(\eta) \circ \Psi
\end{align*}
is an isomorphism.
It is easy to see that $\Psi^{-1}(X,Y,\Lambda) = (X,Y,\Lambda-q(X))$.
Since it is a prism over $G$, we can also compute the generators of $\Lift(\Id_{n-1}\times G)$:
$$\Lift(\Id_{n-1}\times G) = \linsp{\ofu_{n+1}}{\dpar{}{X_1},\ldots,\dpar{}{X_{n-1}},3Y\dpar{}{Y} + 2\Lambda\dpar{}{\Lambda},-\frac{2}{3}\Lambda^2\dpar{}{Y} + 3Y\dpar{}{\Lambda}}$$
Now we can compute the image of these generators via the isomorphism (here we use the fact that $q$ only depends on the first $n-1$ variables and the first $n-1$ components of $\Psi$ are the natural projection):
\begin{align*}
\dpar{}{X_j} & \mapsto \eta^j := \dpar{}{X_j} - \dpar{q}{X_j}(X) \dpar{}{\Lambda} \text{ for } j = 1,\ldots,n-1\\
3Y\dpar{}{Y} + 2\Lambda\dpar{}{\Lambda} & \mapsto \eta^n := 3Y\dpar{}{Y} + 2(\Lambda + q(X))\dpar{}{\Lambda}\\
-\frac{2}{3}\Lambda^2\dpar{}{Y} + 3Y\dpar{}{\Lambda} & \mapsto \eta^{n+1} := -\frac{2}{3}(\Lambda + q(X))^2\dpar{}{Y} + 3Y\dpar{}{\Lambda}\\
\end{align*}


Since $F$ is substantial, this means there exist some $b_1,\ldots,b_{n+1}\in\ofu_{n+1}$ and $\tilde\eta_1,\ldots,\tilde\eta_n\in\ofu_{n+1}$ such that
$$\sum_{j=1}^{n+1} (b_j\cdot\eta^j)(X,Y,\Lambda) = \sum_{j=1}^{n-1}\tilde \eta_j(X,Y,\Lambda) \dpar{}{X_j} +\tilde \eta_n(X,Y,\Lambda) \dpar{}{Y} + \Lambda \dpar{}{\Lambda}.$$
Looking at the coefficient in $\dpar{}{\Lambda}$ we get
$$-\sum_{j=1}^{n-1} b_j\cdot \dpar{q}{X_j}(X) + 2(\Lambda + q(X))b_n + 3Yb_{n+1} = \Lambda. $$
Since $f$ is unstable, $q$ must be a singular function and so $\dpar{q}{X_j}\in\mfr_X$ for all $j=1,\ldots,n-1$.
Therefore $b_n$ must be a unit.
Taking $\Lambda = Y = 0$, it follows that
$$-\sum_{j=1}^{n-1} b_j(X,0,0)\cdot \dpar{q}{X_j}(X) + 2q(X)b_n(X,0,0)  = 0 $$
which means that $q$ belongs to its jacobian ideal.

Now we only need to check that $q$ has isolated singularity, and this will come from the fact that $f$ is $\A$-finite.
If we do this, Saito's theorem in \cite{saito} ensures that $q$ is weighted-homogeneous after some coordinate change in the source.
Let $\tilde \phi\colon(\bbk^{n-1},0)\to(\bbk^{n-1},0)$ be a germ of diffeomorphism such that $\tilde q = q\circ \tilde\phi$ is weighted-homogeneous, let $\phi(x,y) = (\tilde\phi(x),y)$ and $\psi(X,Y) = (\tilde\phi^{-1}(X),Y)$, then
$$\psi\circ f \circ \phi (x,y) = (x,y^3+ \tilde q(x) y)$$
which is weighted-homogeneous.

To check that $q$ must have isolated singularity, we refer to Proposition 2.3 in \cite{marartari}, in which they prove precisely that for an equidimensional map-germ of corank 1 and multiplicity 3, to be $\A$ -finite is equivalent to $q$ being $\K$-finite, which is equivalent to having isolated singularity in the case of functions.
\end{proof}

The last part of the argument is what prevents us from generalizing this approach to other multiplicities.
All corank 1, finite singularity type map-germs in $(\bbk^n,0)\to(\bbk^n,0)$ can be written as
$$f(x,y) = (x, y^{k+1} + q_1(x)y+\cdots + q_{k-1}y^{k-1}).$$
In \cite{marartari} it is proven that  $f$ is $\A$-finite if and only if, under a certain equivalence relation, the tuple $(q_1,\ldots,q_{k-1})$ has finite codimension.
In the case $k=2$, this forces $q$ to be a function with isolated singularity, but for $k>2$ the tuples can have weirder behavior.
This does not seem to make the statement false for higher multiplicities, see \cref{ex_subst_qh} below, but it does mean that some more general argument is needed.

The following example shows that this characterization does not follow for non $\A$-finite map-germs:

\begin{example}\label{ex_subs_no_qh}

Let $f(x,y,z,w) = (x,y,z,w^3 + (xy^3z^3 + y^5+z^5)w)$.
This map-germ is not $\A$-finite.
Its liftable vector fields can be easily computed via SINGULAR and it reveals that it admits no liftable vector field with all positive eigenvalues, therefore it cannot be weighted-homogeneous in any coordinate system.

But the 1-parameter stable unfolding
$$F(x,y,z,w,\lambda) = (x,y,z,w^3 + (\lambda + xy^3z^3 + y^5+z^5)w,\lambda)$$
is substantial, since the vector field $\eta(X,Y,Z,W,\Lambda) = -2X\dpar{}{X}+2Y\dpar{}{Y}+2Z\dpar{}{Z}+15W\dpar{}{W}+10\Lambda\dpar{}{\Lambda} $ is liftable for $F$.
\end{example}

The last example is obtained by augmenting  $w^3$ via the function $g(x,y,z) = xy^3z^3 + y^5+z^5$.
Augmentation is a procedure that allows to construct new singularities as pull-backs of 1-parameter stable unfoldings, while preserving some properties of the map-germs it is obtained from.
For more information on augmentations, we refer to \cite{phiequiv,houstonaug2}.

Here the choice of function seems particularly unfortunate since $g$ does not have isolated singularity.
Moreover, it is not weighted-homogeneous under any coordinate system, since all the Euler vector fields in $\Derlog(g^{-1}(0))$ have mixed positive and negative integer weights.
This function was also used as an example in \cite{bkr_bruce} where they study the relative quasi-homogeneity of functions with respect to given varieties.

Such inconvenient properties of this function make the following example even more interesting:

\begin{example}\label{ex_subst_qh}
Let $f\colon(\bbc^4,0)\to(\bbc^4,0)$ be the map-germ given by $f(x,y,z,w) = (x,y,z,w^4 + xw+(xy^3z^3+y^5+z^5)w^2)$.
This map-germ admits the stable 1-parameter unfolding
$$F(x,y,z,w,\lambda) = (x,y,z,w^4 + xw+(\lambda + xy^3z^3+y^5+z^5)w^2,\lambda)$$
which can be checked that is substantial using SINGULAR.
We omit the specific substantial liftable vector field since the coefficients are too big.
Moreover, we can use this computation and Damon's theorem in \cite{damonakv} to compute that $\aecod(f) = 16$, hence $f$ is $\A$-finite.

In particular, using SINGULAR one can also find the liftable and lowerable vector fields of $F$ and use \cref{prop_lifts_unfolding} to compute those of $f$, finding a pair of $f$-related vector fields with positive integer eigenvalues.
An argument similar to the one in \cref{thm_minimal} shows that there is a coordinate system in which $f$ is weighted-homogeneous.


\end{example}

\section*{Statements and declarations}

\subsection*{Conflict of interests}
The authors declare that there is no conflict of interest.
\subsection*{Data availability} The authors declare that there is no associated data to this manuscript.

\end{document}